\documentclass[a4paper,8pt]{article}

\usepackage{amssymb,latexsym,amsmath,amsthm}
\usepackage{verbatim}
\usepackage{multicol}

\newtheorem{theorem}{Theorem}[section]

\newtheorem{lemma}{Lemma}[section]
\newtheorem{prop}{Proposition}[section]
\newtheorem{df}{Definition}[section]

\newtheorem{question}{Question}[section]

\newtheorem{stmt}{Statement}[section]

\usepackage[colorlinks]{hyperref}
\usepackage{xcolor}
\usepackage{color}

\topmargin by -\headsep
\begin{document}
\title{The Schur-Horn theorem for operators with finite spectrum}
\author{B V Rajarama Bhat, Mohan Ravichandran}
\date{}

\maketitle
\abstract{The carpenter problem in the context of $II_1$ factors, formulated by Kadison asks: Let $\mathcal{A} \subset \mathcal{M}$ be a masa in a type $II_1$ factor and let $E$ be the normal conditional expectation from $\mathcal{M}$ onto $\mathcal{A}$. Then, is it true that for every positive contraction $A$ in $\mathcal{A}$, there is a projection $P$ in $\mathcal{M}$ such that $E(P) = A$?  In this note, we show that this is true if $A$ has finite spectrum. We will then use this result to prove an exact Schur-Horn theorem for (positive)operators with finite spectrum and an approximate Schur-Horn theorem for general (positive)operators. 
\section{Introduction}

Let $\mathcal{A}$ be a masa in a $II_1$ factor $\mathcal{M}$ and $E$ the normal conditional expectation from $\mathcal{M}$ to $\mathcal{A}$. Kadison, in \cite{KadPyt1} asked the following question,  
\begin{question}[Kadison's carpenter problem]
 Given any positive contraction $B$ in $\mathcal{A}$, does there exist a projection $P$ in $\mathcal{M}$ so that $E(P) = B$? 
\end{question}
We will denote the above problem as asking if positive contractions in masas can be lifted to projections. We refer the reader to the above cited paper for the discussion leading up to this problem. The best result to date is the result of \cite{DFHS} that says the following
\begin{prop}[Dykema, Fang, Hadwin, Smith] 
Any positive contraction in a generator masa in $L(F_{2})$ can be lifted to a projection. Also, for any positive contraction $B$ in a Cartan masa $\mathcal{A}$ in the hyperfinite $II_1$ factor $\mathcal{R}$, there is an automorphism $\theta$ of $\mathcal{A}$ so that $\theta(B)$ can be lifted to a projection. 
\end{prop}
There are several consequences of this result that the reader can work out for herself. For general $II_1$ factors, far less is known. Indeed, everything that is known so far with the exception of the result mentioned above and some extensions proved in the same paper, is a straightforward interpretation of results for matrices. For instance, the matricial Schur-Horn theorem guarantees that $\lambda I$ can be lifted if $\lambda$ is a rational number, but it is not known if irrational multiples of the identity can be lifted to projections. In this note we show that this is indeed the case. It will follow that elements with finite spectrum can be lifted to projections.

 In this note, we will work in a slightly more general context. Kadison's carpenter problem is a special case of a majorization problem for von Neumann algebras. The notion of majorization in von Neumann algebras goes back at least to Hiai's work\cite{Hiai} in the 80's.
\begin{df}[Majorization]
Given two self-adjoint operators $A, S$ in a finite factor $(\mathcal{M},\tau)$, say that $A$ is majorized by $S$, denoted by $A \prec S$ if 
\[\tau(f(A)) \leq \tau(f(S))\] for every continuous convex real valued function $f$ defined on a closed interval $[c,d]$ containing the spectra of both $A$ and $S$. 
\end{df}
 The condition implies that $\tau(A) = \tau(S)$. Majorization can be expressed in several ways and these equivalences can be found in \cite{Hiai} and the references therein. A natural extension of Kadison's problem was formulated by Kadison and Arveson in \cite{ArvKad}. 
\begin{question}[Arveson and Kadison's Schur Horn problem]
Let $\mathcal{A}$ be a positive element in $\mathcal{A}$ and $S$ a positive element in $\mathcal{M}$ such that $A \prec S$. Then, is it true that there exists an element $T$ in $\mathcal{O}(S) = \overline{\{U S U^{*},\,\, U \in \mathcal{U}(M)\}}^{||}$ such that $E(T) = A$?  
\end{question}
One does need to take the norm closure; See the example following lemma(5.5) in the same paper.

This problem was solved in the affirmative for the generator and radial masas in the free group factors in \cite{DFHS}, where it was also solved modulo an automorphism of the masa for Cartain masas in the hyperfinite $II_1$. In this note, we will work with general masas inside general type $II_1$ factors. Our main result is the following theorem whose proof is an adaptation of the best known proof of the matricial Schur Horn theorem. It should come as no surprise that we do not need to take the norm closure to achieve lifting. 
\begin{theorem}
 Let $\mathcal{A}$ be a masa in a $II_1$ factor $\mathcal{M}$ and let $E$ be the normal conditional expectation from $\mathcal{M}$ to $\mathcal{A}$. Let $A \in \mathcal{A}$ and $S \in \mathcal{M}$ be positive operators with finite spectrum such that $A \prec S$. Then, there is a unitary $U$ in $\mathcal{M}$ so that $E(U S U^{*}) = A$. 
\end{theorem}

 The theorem says that the Schur-Horn problem can be solved when both elements have finite spectrum. While this result will hardly come as a surprise, it is new. Routine calculations will then allow us to adapt the above theorem to deduce an approximate Schur-Horn theorem for general operators in a $II_1$ factor.

\begin{theorem}
Let $S$ be a self-adjoint operator in $\mathcal{M}$. Then, the norm closure of $E(\mathcal{U}(S))$ equals $\{A \in \mathcal{A} \mid A \prec S\}$.
\end{theorem}
In particular, letting $\mathcal{O}(S) = \overline{\{USU^{*} \mid U \in \mathcal{U}(\mathcal{M})\}}^{||}$, we have that
\[\overline{E(\mathcal{O}(S))}^{||} = \{A \in \mathcal{A} \mid A \prec S\}\]

The conjectured Schur-Horn theorem of Arveson and Kadison says that we do not need to take the norm closure for equality, something that we are unable to prove in this note. A weaker version of our theorem, where the $\sigma-$SOT closure was taken in the place of the norm closure was proved by Argerami and Massey in \cite{MasArgIn}. Also, the above result was established for Cartan masas in the hyperfinite $II_1$ factors(and thus for general semi-regular masas, see \cite{PopKad}) in \cite{DFHS}.

The paper has four sections apart from the introduction; In section $2$, we show that scalars can be lifted to projections. In section $3$, we push this through to show that the Schur-Horn problem can be solved for operators with finite spectrum. Section $4$ contains the approximate Schur-Horn theorem. There is then a last section consisting of some remarks and observations. 

Some words on notation: Given two operators $A, B$ inside a von Neumann algebra $\mathcal{M}$ such that there is a projection $P$ inside $\mathcal{M}$ such that $A = PAP$ and $B = (I-P)B(I-P)$, in order to stress the fact that $A$ and $B$ live under the auspices of orthogonal projections, we will use the expression $A\oplus B$ to denote their sum. Next, given a self-adjoint operator $A$ and a Borel measurable subset $X$ of the real line, the expression $E_{A}(X)$ will denote the spectral projection of $A$ corresponding to the subset $X$. This notation might cause confusion with the notation $E_{\mathcal{A}}(A)$ or simply $E(A)$ where $\mathcal{A}$ is a subalgebra of $\mathcal{M}$, which denotes the image under a conditional expectation $E$. We apologize for this, but retain the notations due to their provenance. Finally, lower case letters, possibly with subscripts, like $a, b$ and $s_i$ will always refer to scalars. We will always use upper case letters $S, T$ and so forth to refer to operators.

\section{Lifting Scalars}
We begin with a simple observation. 
\begin{lemma}\label{simLem}
Let $P$ be a projection in a masa $\mathcal{A}$ inside a type $II_1$ factor $\mathcal{M}$ and let $\lambda, a, b$ be positive scalars such that $\tau(S) = \lambda$ where $S = a P + b (I-P)$. Then, there is a unitary $U$ in $\mathcal{M}$ and a projection $Q$ in $\mathcal{A}$ such that letting $T = U S U^{*}$, we have that
\begin{enumerate}
 \item $E(Q T Q) = \lambda Q$.
 \item $(I-Q) T (I - Q) = c R + d (I - Q - R)$ for some projection $R$ in $\mathcal{A}$ with $R \leq I - Q$ and positive numbers $c, d$. 
 \item $\tau(Q) \geq \dfrac{1}{3}$.
\end{enumerate}
   
\end{lemma}
\begin{proof}
The lemma is trivial if $a = b$, for then, $a = b = \lambda$ and there is nothing to prove. We assume without loss of generality that $a > b$. Since $\tau(S) = \lambda$, we must then have that $a > \lambda > b$. We may also assume that $\tau(P) \leq \dfrac{1}{2}$. For, suppose we have proved the lemma in this case, the result when $\tau(P) > \dfrac{1}{2}$ can be derived by applying the lemma to $I-S$ and $(1 -\lambda)I$. We therefore assume that $\tau(P) \leq \dfrac{1}{2}$. 

 Let $k$ be the largest integer such that $(k+1) \tau(P) \leq 1$. Since $\tau(P) \leq \dfrac{1}{2}$, $k$ must be at least $1$. Pick projections $Q_1, \cdots, Q_{k}$, each of trace $\tau(P)$ in $\mathcal{A}$ that are mutually orthogonal and also orthogonal to $P$. let $V_{1}, \cdots, V_{k}$ be  partial isometries in $\mathcal{M}$ such that 
\begin{enumerate}
 \item $V_{1}^{*} V_{1} = Q_{1}$ and $V_{1} V_{1}^{*} = P$ .
 \item For $2 \leq i \leq k$, $V_{i} V_{i}^{*} = Q_{i-1}$ and  $V_{i}^{*} V_{i} = Q_{i}$. 
 
\end{enumerate}
Pick $\theta_{1}$ such that $a \operatorname{cos}^{2}(\theta_{1}) + b \operatorname{sin}^{2}(\theta_1) = \lambda$ and let $U_1$ be the operator 
\[U_1 = \operatorname{cos}(\theta_{1}) P + \operatorname{sin}(\theta_{1}) V_1 - \operatorname{sin}(\theta_{1}) V_1^{*}  +  \operatorname{cos}(\theta_{1}) Q_{1} + (I - P - Q_1)\]
 We will identify the above operator with the operator matrix(using $V_1$ as the matrix unit $E_{12}$), an identification that is standard. 

\[U_1 =  \left( \begin{array}{ccc}
\operatorname{cos}(\theta_{1}) & \operatorname{sin}(\theta_{1}) & 0\\
-\operatorname{sin}(\theta_{1}) & \operatorname{cos}(\theta_{1}) & 0\\ 
0 & 0 & I\end{array} \right)\] 

In this same identification, $S$ is the operator
\[S =  \left( \begin{array}{ccc}
a & 0 & 0\\
0 & b & 0\\ 
0 & 0 & b\end{array} \right)\] 

Let $S_{1} = U_1 S U_{1}^{*}$. It is easy to check that $U_1$ is a unitary and that
 
\[ S_1 = \left( \begin{array}{ccc}
a \operatorname{cos}^{2}(\theta_{1}) + b \operatorname{sin}^{2}(\theta_1) & \ast & 0\\
\ast & a \operatorname{sin}^{2}(\theta_{1}) + b \operatorname{cos}^{2}(\theta_{1}) & 0\\ 
0 & 0 & b\end{array} \right)  =  \left( \begin{array}{ccc}
\lambda & \ast & 0\\
\ast & a_1 & 0\\ 
0 & 0 & b_1\end{array} \right)\]
where $a_1 = a \operatorname{sin}^{2}(\theta_{1}) + b \operatorname{cos}^{2}(\theta_{1})$ and $b_1 = b$. By the trace condition, 
\[\lambda \tau(P) + a_1 \tau(P) + b_1 (1 - 2 \tau(P)) = \lambda.\] Since $b_1 = b < \lambda$, we must have that $a_1 > \lambda$ and afortiori $a_1 > b_1$. 

Now, continue as above. Pick $\theta_{2}$ such that $a_1 \operatorname{cos}^{2}(\theta_{2}) + b_1 \operatorname{sin}^{2}(\theta_2) = \lambda$ and let $U_2$ be the operator 
\[U_2 = \operatorname{cos}(\theta_2) Q_1 + \operatorname{sin}(\theta_2) V_2 - \operatorname{sin}(\theta_2) V_2^{*}  +  \operatorname{cos}(\theta_2) Q_2 + (I - Q_1 - Q_2).\] We may write the unitary $U_2$ as
\[U_2 =  \left( \begin{array}{cccc}
I & 0 & 0 & 0\\
0 &\operatorname{cos}(\theta_{2}) & \operatorname{sin}(\theta_{2}) & 0\\
0 & -\operatorname{sin}(\theta_{2}) & \operatorname{cos}(\theta_{2}) & 0\\ 
0 & 0 & 0 & I\end{array} \right)\]
and let $S_2 = U_{2} S_{1} U_{2}^{*}$. We have that
\[ S_2 = \left( \begin{array}{cccc}
\lambda & \ast & \ast & 0\\
\ast & \lambda & \ast & 0\\
\ast & \ast & a_2 & 0\\ 
0 & 0 & 0 & b_2 \end{array} \right)\]
where $a_2 = a_1 \operatorname{sin}^{2}(\theta_{2}) + b_1 \operatorname{cos}^{2}(\theta_{2})$ and $b_2 = b_1$. By the trace condition, 
\[2\lambda \tau(P) + a_2 \tau(P) + b_2 (1 - 3 \tau(P)) = \lambda.\] Since $b_2 = b_1 = b < \lambda$, we must have that $a_2 > \lambda$ and afortiori $a_2 > b_2$.

Proceeding this, $k-2$ more times, we get an operator $S_k$ of the form 
\[ S_k = \left( \begin{array}{cccccc}
\lambda & \ast & \hdots & \ast & \ast & 0\\
\ast & \lambda & \hdots & \ast & \ast & 0\\
\vdots & \vdots & \ddots & \vdots & \ast & 0\\
\ast & \ast & \ast & \lambda & \ast & 0\\
\ast & \ast & \ast & \ast & a_k & 0  \\
0 & 0 & 0 & 0 & 0 & b_k\end{array} \right)\]
Let $Q = P + Q_{1} + \cdots + Q_{k-1}$(if $k = 1$, let $Q = P$). We see that
\begin{enumerate}
 \item $E(Q S_{k} Q) = \lambda Q$. This is because, 
\begin{eqnarray*}
E(Q S_{k} Q) &=& E(P S_{k} P) + E(Q_1 S_k Q_1) + \cdots + E(Q_{k-1} S_k Q_{k-1})\\
             &=& \lambda P + \lambda Q_1 + \cdots + \lambda Q_{k-1}\\
	     &=& \lambda Q.
\end{eqnarray*}
($S_k$ is the operator $T$ promised in the statement of the lemma).
 \item $(I - Q) S_{k} (I - Q)$ has two point spectrum in $(I-Q) M (I- Q)$. 
 \item $\tau(Q) = k \tau(P)$. Since $(k+1) \tau(P) \leq 1 < (k+2) \tau(P)$, we see that 
\[\tau(Q) = k \tau(P) = \dfrac{k}{k+2} (k+2) \tau(P) > \dfrac{k}{k+2} \geq \dfrac{1}{3}.\] 
\end{enumerate}
The lemma follows. 
\end{proof}

\begin{theorem}
Let $\mathcal{A}$ be a masa in a $II_1$ factor $\mathcal{M}$ and let $E$ be the normal conditional expectation from $\mathcal{M}$ to $\mathcal{A}$. Then for any $0 \leq \lambda \leq 1$, there is a projection $P$ in $\mathcal{M}$ such that $E(P) = \lambda I$. 
\end{theorem}
\begin{proof}
Let $P_0$ be any projection of trace $\lambda$ in $\mathcal{A}$. Using lemma(\ref{simLem}), construct a unitary $U_1$ and a projection $Q_1$ in $\mathcal{A}$ such that, letting $P_1 = U_1 P_0 U_1^{*}$,
\begin{enumerate}
 \item  $\tau(Q_1) \geq \dfrac{1}{3}$. 
 \item $E(Q_1 P_1 Q_1) = \lambda Q_1$
 \item $(I-Q_1) P_1 (I - Q_1)$ has two point spectrum in $(I-Q_1) M (I - Q_1)$.  
\end{enumerate}
Let $R_{1} = Q_1$. Next, for $k =  2, 3, \cdots$, apply lemma(\ref{simLem}) to $\lambda (I-R_{k-1})$ and $(I-R_{k-1}) Q_{k-1} (I - R_{k-1})$ inside the $II_1$ factor $(I-R_{k-1})M(I-R_{k-1})$ to construct a unitary $U_{k}$ and a projection $Q_{k}$ in $(I-R_{k-1}) M (I - R_{k-1})$ and let
\[R_{k} = Q_1 \oplus Q_2 \oplus \cdots \oplus Q_{k} \quad \operatorname{and} \quad  P_{k} = (R_{k-1} \oplus U_{k})  P_{k-1} (R_{k-1} \oplus U_{k})^{*}\]Here we identify $Q_{k}$ which is a projection in  $(I-R_{k-1}) M (I - R_{k-1})$ with a projection in $\mathcal{M}$ dominated by $I-R_{k-1}$. Also note that $P_k$ is a projection. We have that
\begin{enumerate}
 \item $E(Q_{k} P_{k} Q_{k}) = \lambda Q_{k}$ and thus, 
\[E(R_{k} P_{k} R_{k}) = \sum_{m=1}^{k} E(Q_{m} P_{m} Q_{m}) = \sum_{m=1}^{k} \lambda Q_{m} = \lambda R_{k}.\]
 \item $\tau(I - R_{k}) \leq \dfrac{2}{3} \tau(I - R_{k-1}) \leq (\dfrac{2}{3})^{k}$ and hence, $R_{k}$ converges to $I$ strongly.
 \item $(I-R_{k}) P_{k} (I -R_{k})$ has two point spectrum in $(I-R_k) M (I - R_k)$. 
 \item We have that $R_{k-1} P_{k-1} R_{k-1} = R_{k-1} P_{k} R_{k-1}$ and thus, 
\begin{eqnarray}\label{constancy}
R_{l} (P_{m} - P_{n}) R_{l} = 0 \quad \text{ for any } \quad n,m \geq l.
\end{eqnarray}
\end{enumerate}

We now claim that $P_{k}$ converges in the strong operator topology to a projection that we will call $P$ and also that $E(P) = \lambda I$. For the first claim, since $\tau(R_k)$ converges strongly to $I$, for any $\epsilon > 0$ there is a $N$ so that $||(I - R_N) ||_{2} < \epsilon$. For $n, m \geq N$, 
\[  ||(P_{n} - P_{m})||_{2} \leq ||R_{N}(P_{n} - P_{m}) R_{N}||_{2} + 2||(I-R_{N})(P_{n} - P_{m})||_{2}\]

The first term is zero by (\ref{constancy}). For the second term,
\[||(I-R_{N})(P_{n} - P_{m})||_{2} \leq ||I - R_{N}||_{2} ||P_{n} - P_{m}|| \leq 2 \epsilon\]
Thus, $||(P_{n} - P_{m})||_{2} \leq 4 \epsilon$ and the sequence $\{P_n\}$ is strongly convergent. Let $P$ be the limit projection. Forthe second claim, 
\begin{align*}
||E(P) - \lambda I||_{2} = &\operatorname{lim} ||E(P_{n}) -\lambda I||_{2}\\
= &\operatorname{lim} ||\lambda R_{n} + E((I-R_{n}) P_{n} (I-R_{n})) - \lambda I||_{2}\\
= &\operatorname{lim} ||-\lambda(I - R_{n}) + E((I-R_{n}) P_{n} (I-R_{n})||_{2}\\
 \leq &\operatorname{lim}\lambda||(I-R_{n}||_{2} + ||(I-R_{n})P_n(I-R_{n})||_{2}\\
 \leq &\operatorname{lim}\lambda \left(\dfrac{2}{3}\right)^{n} + ||P_n|| \left(\dfrac{2}{3}\right)^{n} \\
 \leq &\operatorname{lim}(\lambda + 1)\left(\dfrac{2}{3}\right)^{n} \\
 = & \,0
\end{align*}
We conclude that $E(P) = \lambda I$.
\end{proof}

We record a simple corollary
\begin{prop}\label{prop00}
Let $A$  be a positive contraction in $\mathcal{A}$ that can be written as $A = \sum_{n} \lambda_n E_{n}$, where the $E_{n}$'s are orthogonal projections summing up to $I$. Then, there is a projection $P$ in $\mathcal{M}$ such that $E(P) = A$. 
\end{prop}
\begin{proof}
The element $A$ may be written as $A = \sum_{n=1}^{\infty} \lambda_n E_n$ where the $E_{n}$'s are mutually orthogonal projections in $\mathcal{A}$ summing up to $1$ and $0 \leq \lambda _n \leq 1$ for every $n$. $E_{n} M E_{n}$ is a type $II_1$ factor and we may find a projection $P_{n}$ in $E_{n} M E_{n}$ such that $E_{AE_{n}}(P_{n}) = \lambda_{n} E_{n}$ for every $n$. Let $P$ be the projection $\sum_{n=1}^{\infty} P_{n}$. Here, we are identifying $P_{n}$ which is a projection in $E_{n} M E_{n}$ with a projection in $\mathcal{M}$ that is dominated by $E_{n}$. Then,
\[E(P) = \sum_{n=1}^{\infty} E(P_{n}) = \sum_{n=1}^{\infty} E(E_{n} P_{n} E_{n}) = \sum_{n=1}^{\infty} \lambda_{n} E_{n} = A\]
\end{proof}

\section{Schur-Horn theorem for operators with finite spectrum}
We will now bootstrap the theorem in the previous section to get a Schur-Horn theorem for positive operators with finite spectrum. Recall the following reformulation of majorization in $II_1$ factors. Let $A, S$ be positive contractions in a type $II_1$ factor $\mathcal{M}$ and let $f, g : [0,1] \rightarrow [0,1]$ be the (essentially unique, right-continuous, non-increasing) spectral weight functions, which satisfy
\[\tau(A^{n}) = \int_{0}^{1} f^{n}(r) dm(r) \quad \text{ and } \quad \tau(S^{n}) = \int_{0}^{1} g^{n}(r) dm(r) \quad \text{for } n =0, 1, \cdots\] 
Then $A \prec S$ if 
\[\int_{0}^{t} f(r) dm(r) \leq \int_{0}^{t} g(r)dm(r), \,\, 0 \leq t \leq 1 \quad \text{and} \quad \int_{0}^{1}f(r)dm(r) = \int_{0}^{1} g(r)dm(r) \]

\begin{lemma}\label{2ptSp}
 Let $A = \lambda_1 E_1 \oplus \lambda_2 E_2$ where $E_1 + E_2 = I$ and $\lambda_1 \geq \lambda_2 \geq 0$ and $S = \mu_1 F_1 \oplus \mu_2 F_2$ where $F_1 + F_2 = I$ and $\mu_1 > \mu_2 \geq 0$ be two operators in a $II_1$ factor with $\tau(A) = \tau(S)$. If $\mu_1 \geq \lambda_1$ and $\mu_2 \leq \lambda_2$, then $A \prec S$.  
\end{lemma}
\begin{proof}

It is easy to see that if $B$ is a positive contraction, then $B \prec P$ for any projection $P$ with $\tau(P) = \tau(B)$. Let $c = \dfrac{1}{\mu_1 - \mu_2}$ and $d =  - \dfrac{\mu_2}{\mu_1 - \mu_2}$. The operator $cS + dI$ may be checked to equal $F_1$, is hence a projection and of course, $\tau(cS+dI) = \tau(cA+dI)$. 
\[cA + dI = (c \lambda_1 +d)E_{1} + (c \lambda_2 + d)E_{2} = \dfrac{\lambda_1-\mu_2}{\mu_1-\mu_2} E_{1} +  \dfrac{\lambda_2-\mu_2}{\mu_1-\mu_2} E_{2}\]

Since $\lambda_2 \leq \lambda_1$, $\lambda_2 \geq \mu_2$ and $\lambda_1 \leq \mu_1$, we have that 
\[0 \leq \dfrac{\lambda_2-\mu_2}{\mu_1-\mu_2} \leq \dfrac{\lambda_1-\mu_2}{\mu_1-\mu_2} \leq \dfrac{\mu_1-\mu_2}{\mu_1-\mu_2} = 1\]
And thus, $cA + dI$ is a positive contraction. By the observation in the first line of the proof, $cA + dI \prec cS + dI$ and therefore, $A \prec S$. 
\end{proof}

\begin{lemma}\label{2ptMajCond}
 Let $A = \lambda_1 E_1 + \lambda_2 E_2$ and $S = \mu_1 E_1 + \mu_2 E_2$ where $E_1$ and $E_2$ are orthogonal projections summing up to $I$, be positive operators in a type $II_1$ factor $\mathcal{M}$, with the same trace. If $\lambda_1 \leq \mu_1$, then $A \prec S$. 
\end{lemma}
\begin{proof}
 It is easy to see that we must have $\mu_2 < \lambda_2$. The lemma now follows from lemma(\ref{2ptSp}).
\end{proof}

\begin{lemma}\label{2ptMaj}
Let $A$ be a self-adjoint operator and $S$ a positive contraction in a $II_1$ factor so that $A \prec S$. Then $A$ is a positive contraction as well. 
\end{lemma}
\begin{proof}
 Routine verification.
\end{proof}

\begin{prop}\label{prop0}
Let $S$ be a positive operator in $\mathcal{M}$ with two point spectrum and let $A$ be a positive contraction in $\mathcal{A}$ that has finite spectrum and so that $A \prec S$. Then, there is a unitary $U$ in $\mathcal{M}$ such that $E(U S U^{*}) = A$. 
\end{prop}
\begin{proof} 
Write $S = \mu_1 F_1 \oplus \mu_2 F_2$ where $\mu_1 \geq \mu_2$ and $F_1 \oplus F_2 = I$. Let $c = \dfrac{1}{\mu_1 - \mu_2}$(note that $c>0$) and $d =  - \dfrac{\mu_2}{\mu_1 - \mu_2}$. The operator $cS + dI$ may be checked to equal $F_1$ and is hence a projection. We also have that $cA+dI \prec cS+dI = F_1$. By lemma(\ref{2ptMaj}), $cA+dI$ must actually be a positive contraction. Also, of course, $\tau(cS+dI) = \tau(cA+dI)$. Now, by proposition(\ref{prop00}), there is a unitary $U$ so that $E(U(cS+dI)U^{*}) = cA+dI$. And hence, $E(USU^{*})=A$.

\end{proof}

When one or both operators have finite spectrum, majorization reduces to a simple condition.
\begin{lemma}\label{atoMaj}
 Let $A, S$ be positive operators in a $II_1$ factor with $\tau(A) = \tau(S)$ and let $f, g$ be the spectral weight functions of $A, S$ respectively, as above. Suppose $A$ has finite spectrum, i.e, the spectral weight function $f$ has the form
\[f = \sum_{n=1}^{N} \lambda_n \chi_{[s_{n-1},s_n)}\]
for some natural number $N$ and some sequences $0 = s_0 < s_1 < \cdots < s_N = 1$ and $0 \leq \lambda_N < \lambda_{N-1} < \cdots < \lambda_1$. Then, $A \prec S$ iff for $n = 1, 2, \cdots, N$,
\[\int_{0}^{s_n} f(r) dm(r) \leq \int_{0}^{s_n} g(r)dm(r) \,\, \text{ or equivalently, } \,\, \tau(A E_{A}([0,s_n))) \leq \tau(S E_{S}([0,s_n))) \]
\end{lemma}
\begin{proof}
Routine verification. 

\end{proof}

We now prove the promised special case of the Schur-Horn theorem.

\begin{theorem}[The Schur-Horn theorem for operators with finite spectrum in a $II_1$ factor]\label{SHD}
Let $A$ and $S$ be positive operators with finite spectrum in $\mathcal{A}$ and $\mathcal{M}$ respectively and so that $A \prec S$. Then, there is a unitary $U$ in $\mathcal{M}$ so that $E(U S U^{*}) = A$. 
\end{theorem}
\begin{proof}
We assume that $A$ and $S$ have spectrum consisting of $N$ and $M$ points respectively. Write $A = \sum_{n=1}^{N} \lambda_n E_n$ and $S = \sum_{n=1}^{M} \mu_n F_n$ where the $\{\lambda_n\}_{1}^{N}$(respectively, the $\{\mu_n\}_{1}^{M}$) are distinct. We may  assume that none of the $\lambda_i$ equal any of the $\mu_j$. For suppose $\lambda_i = \mu_j$. Assume that $\tau(E_i) \leq \tau(F_j)$, the other case is handled similarly.  We may, after conjugating by a unitary, write $A = \lambda_i E_i \oplus (A-\lambda_i E_i)$ and $S = \mu_j E_i \oplus (S - \mu_j E_i) = \lambda_i E_i \oplus (S - \lambda_i E_i)$. Clearly, $A-\lambda_i E_i \prec S - \lambda_i E_i$ and it is enough to prove the theorem for $A-\lambda_i E_i$ which has at most $N-1$ point spectrum in $(I-E_i)\mathcal{A}$ and  $S - \lambda_i E_i$ which has at most $M$ point spectrum inside $(I-E_i) \mathcal{M} (I-E_i)$. We therefore assume that none of the $\lambda_i$ equal any of the $\mu_j$.

Since $\mathcal{A}$ is unitarily equivalent to $L^{\infty}([0,1],dm)$, we may find a maximal nest of projections $\{P_{t} : 0 \leq t \leq 1\}$ in $\mathcal{A}$ with $P_{t} \leq P_{s}$ for $0 \leq t \leq s \leq 1$ and $\tau(P_{t}) = t$ for $0 \leq t \leq 1$. Since $A$(respectively $S$) has $N$(respectively $M$) point spectrum, we may, after conjugating $A$ and $S$ by unitaries, assume that $A$ and $S$ have the form
\[A = \sum_{n=1}^{N} \lambda_n (P_{s_{n}}-P_{s_{n-1}}) \quad \text{ and } \quad S = \sum_{n=1}^{M} \mu_n (P_{t_{n}} - P_{t_{n-1}})\]
for sequences $0 = s_{0} < s_{1} < s_{2} < \cdots < s_{N} = 1$ and $0 = t_{0} < t_{1} < t_{2} < \cdots < t_{M} = 1$ and positive scalars $\lambda_1 > \lambda_2 > \cdots > \lambda_N \geq 0$ and $\mu_1 > \mu_2 > \cdots > \mu_{M} \geq 0$.

 Reindex the set $\{s_{1}, \cdots, s_{N-1}\} \cup \{t_1, \cdots, t_{M-1}\}$ by $\{r_1, \cdots, r_{L}\}$ where $r_1 < r_2 < \cdots < r_{L-1}$ and let $r_{L} = 1$. Then, we may write 
\[A = \sum_{n=1}^{L} \gamma_n (P_{r_{n}} - P_{r_{n-1}}) \quad \text{ and } \quad f = \sum_{n=1}^{L} \delta_n (P_{r_{n}} - P_{r_{n-1}})\]
where $\gamma_n = \lambda_m$ for the unique value $m$ so that $[r_{n-1},r_{n}) \subset [s_{m-1},s_{m})$ and similarly for the numbers $\delta_n$.

We will prove the theorem by induction on $L$. When $L=1$, $A$ and $S$ are scalars and thus, $A = S = \tau(A)I$ and the theorem is trivial. Assume we have shown the following:
\begin{stmt}
Let $A$ and $S$ be positive operators inside a masa, which we denote by $\mathcal{A}$ inside a type $II_1$ factor, which we denote by $\mathcal{M}$, so that $A = \sum_{n=1}^{K} \gamma_{n} (P_{r_{n}} - P_{r_{n-1}})$ and $S = \sum_{n=1}^{K} \delta_{n} (P_{r_{n}} - P_{r_{n-1}})$ for some sequences $0 < r_{1} < \cdots <r_{K-1} < r_{K} = 1$, $\gamma_{1} \geq \cdots \geq \gamma_{K}$, $\delta_{1} \geq \cdots \geq \delta_{K}$, where $K$ is a natural number less than $L$. Then, there is a unitary $U$ so that $E(U S U^{*}) = A$.
\end{stmt}

We will now show that we can extend this to the case when the decompositions have length $L$ as well. 

The majorization condition for the operators $A$ and $S$ that we are working with becomes the following: $\tau(A) = \tau(S)$ and for every $k = 1, \cdots, L-1$, we have that
\[\int_{0}^{r_{k}} f(r)dm(r) = \sum_{n=1}^{k} \gamma_n (r_{n}-r_{n-1})  \leq \sum_{n=1}^{k} \delta_n (r_{n}-r_{n-1}) = \int_{0}^{r_{k}} g(r)dm(r)\]
 In particular, $\gamma_1 < \delta_1$. If $\gamma_n < \delta_n$ for every $n = 1, \cdots, L$, then, 
\[\tau(A) = \sum_{n=1}^{L} \gamma_n (r_{n}-r_{n-1}) <  \sum_{n=1}^{L} \delta_n (r_{n}-r_{n-1}) = \tau(S)\]
which contradicts the fact that $A \prec S$(which entails that $\tau(A) = \tau(S)$. Thus, there is a natural number $1 < l \leq L$ so that 
\[\gamma_n < \delta_n \,\, \text{ for } \,\, n = 1, \cdots, l \quad \text{ and } \quad \gamma_{l+1} > \delta_{l+1}.\]
 Suppose that $(\delta_{l} - \gamma_{l}) (r_{l}-r_{l-1}) > (\gamma_{l+1} - \delta_{l+1}) (r_{l+1}-r_{l})$(the other case is handled similarly). Pick $r$ so that $(\delta_{l} - \gamma_{l}) (r_{l}-r_{l-1}) = (\gamma_{l+1} - \delta_{l+1}) (r-r_{l})$. Let 
\[A_1: = \gamma_{l} (P_{r_{l}} - P_{r_{l-1}}) + \gamma_{l+1} (P_{r}-P_{r_{l}}) \quad \text{ and } \quad S_1:= \delta_{l} (P_{r_{l}} - P_{r_{l-1}}) + \delta_{l+1} (P_{r}-P_{r_{l}})\]
Then,
\[\tau(S_1 - A_1) = (\delta_l - \gamma_l)(r_{l} - r_{l-1}) + (\delta_{l+1}-\gamma_{l+1})(r - r_{l}) = 0 \]
 Combining this with the fact that $\gamma_l < \delta_l$ and using lemma(\ref{2ptMajCond}), we conclude that 
\[A_{1} \prec S_{1}\] inside the $II_1$ factor $P\mathcal{M}P$ where $P$ is the projection $P = P_{r}-P_{r_{l-1}}$. Now, let 
\[A_{2}:= A - A_{1} = \sum_{n\neq l,l+1} \gamma_{n} (P_{r_{n}} - P_{r_{n-1}}) + \gamma_{l+1} (P_{r_{l+1}}-P_{r})\] and similarly, 
\[S_{2} := S - S_{1} =  \sum_{n\neq l,l+1} \delta_{n} (P_{r_{n}} - P_{r_{n-1}}) + \delta_{l+1} (P_{r_{l+1}}-P_{r}) \]
where the operators are considered in $(I-P)\mathcal{M}(I-P)$. We have
\begin{enumerate}
 \item $\sum_{n=1}^{k} \gamma_n (r_{n}-r_{n-1}) <  \sum_{n=1}^{k} \delta_n (r_{n}-r_{n-1})$ for $k = 1, \cdots, l-1$. 
 \item And for $k \geq l+1$, (if $k = l+1$, the third term in the first expression below will not show up)
\begin{eqnarray*}
 &&\sum_{n=1}^{l-1} \gamma_{n} (r_{n}-r_{n-1}) + \gamma_{l+1} (P_{r_{l+1}}-P_{r}) + \sum_{n=l+2}^{k} \gamma_{n} (r_{n}-r_{n-1})\, = \, \sum_{n=1}^{l+1} \gamma_{n} (r_{n}-r_{n-1})  \\ 
 &&< \sum_{n=1}^{l+1} \delta_{n} (r_{n}-r_{n-1})\\
 &&= \sum_{n\neq l,l+1} \gamma_{n} (P_{r_{n}} - P_{r_{n-1}}) + \gamma_{l+1} (r_{n}-r_{n-1}) + \sum_{n=l+2}^{L-1} \delta_{n} (r_{n}-r_{n-1})
\end{eqnarray*}
since $(\gamma_{l} - \delta_{l}) (r_{l}-r_{l-1}) + (\gamma_{l+1} - \delta_{l+1}) (r-r_{l}) = 0$.
\item $\tau(A_2) = \tau(A) - \tau(A_{1}) = \tau(S) - \tau(S_{1}) = \tau(S_{2})$.
\end{enumerate}
We thus conclude that we also have that 
\[A_{2} \prec S_{2}\]

By proposition(\ref{prop0}), there is a unitary $U_1$ inside $P\mathcal{M}P$ so that $E(U_{1} S_{1} U_{1}^{*}) = A_{1}$. Also, the induction hypothesis holds for the operators $A_{2}$ and $M_{2}$ inside $(I-P)\mathcal{M}(I-P)$ since the partition decomposition for $A_2$ and $S_2$ has length $L-1$. We may therefore find a unitary $U_2$ inside  so that $E(U_{2} S_{2} U_{2}^{*}) = A_{2}$. Thus, letting $U = U_{1} \oplus U_{2}$, we have that $E(U S U^{*}) = A$. 
\end{proof}

\section{An approximate Schur-Horn theorem}

Theorem(\ref{SHD}) allows us to prove an approximate version of the Schur-Horn theorem for general operators. 
\begin{theorem}
 Let $S$ be a positive operator in a $II_1$ factor $\mathcal{M}$ and let $\mathcal{A}$ be a masa in $\mathcal{M}$. Then, the norm closure of $E(\mathcal{U}(S))$ equals  $\{A \in \mathcal{A}^{+} \mid A \prec S\}$. 
\end{theorem}
\begin{proof}
Choose $A$ in $\mathcal{A}^{+}$ so that $A \prec S$. By scaling, if needed, we assume that $A$ and $S$ are strict contractions. Fix $n > 0$ and define the mutually orthogonal projections 
\[P_{k} = E_{A}([\frac{k-1}{n},\frac{k}{n})) \quad \text{ for } \quad 1 \leq k \leq n\]
Next, define $\alpha_k = \tau(AP_{k})$ for $1 \leq k \leq n$ and consider the operator $B = \sum_{k=1}^{n} \alpha_k P_{k}$.  Since $\tau(C) I \prec C$ for any positive operator $C$, we have that $B \prec A$ and hence, $B \prec S$. We also have that
\[||A - B|| = ||\sum_{k=1}^{n} (A-\alpha_k) P_{k}|| \leq ||\sum_{k=1}^{n} \dfrac{1}{n} P_{k}|| = \dfrac{1}{n}\]
Choose numbers $0 = t_0 \leq t_1 \leq t_2 \cdots \leq t_n = 1$ and orthogonal projections $Q_1, \cdots, Q_n$ in $\{S^{'} \cap \mathcal{M}\}$ such that $Q_{k} \leq E_{S}([t_{k-1},t_{k}])$ and $\tau(Q_{k}) = \tau(P_{k})$ for $1 \leq k \leq n$. To see why this possible, proceed thus: Let $t_0 = 0$ and pick $t_1$ such that $\tau( E_{S}([0,t_{1})))\leq \tau(P_{1}) \leq \tau( E_{S}([0,t_{1}]))$. If $S$ has no atom at $t_1$, then let $Q_1 =  E_{S}([0,t_{1}))$. If $S$ has an atom at $t_1$, pick a subprojection $R$ of $E_{S}(\{t_1\})$ such that $\tau(E_{S}([0,t_{1}))) + \tau(R) = \tau(P_1)$ and let $Q_1 = E_{S}([0,t_{1})) + R$. Continue this process for $n=2, \cdots$. 

Next, pick positive operators $T_{1}, \cdots, T_{n}$ all with finite spectrum such that for $1 \leq k \leq n$,
\[ T_{k} \prec S Q_{k}, \quad \tau(T_{k}) = \tau(S Q_{k}) \quad \text{ and } ||S Q_{k} - T_{k}|| \leq \dfrac{1}{n}\] 
This is done exactly in the same way as the choice of the operator $B$ given the operator $A$, in the first part of this proof. Let $T$ be the operator $T = T_{1} + \cdots + T_{n}$. Then, the above conditions imply
\[T \prec S\quad \text{ and } \quad||S - T|| \leq \dfrac{1}{n} \]
Also, for $1 \leq k \leq m$,
\[\tau(T(Q_1 + \cdots + Q_k)) = \tau(S(Q_1 + \cdots + Q_k)) \geq \tau(A (P_1 + \cdots P_k)) = \tau(B (P_1 + \cdots P_k))\]
 and hence, by lemma(\ref{atoMaj}) $B \prec T$. Since $B$ and $T$ have finite spectrum, there is a unitary $U$ so that $B = E(UTU^{*})$. We calculate,
\[||A - E(USU^{*})|| \leq ||A-B|| + ||B - E(UTU^{*}|| + ||E(UTU^{*} - USU^{*})|| \leq \dfrac{1}{n} + 0 + \dfrac{1}{n}\]
and see that $A$ can be arbitrarily well approximated by elements in $E(\mathcal{U}(S))$. Since $A$ was arbitrary, we have that the norm closure of $E(\mathcal{U}(S))$ equals $\{A \in \mathcal{A}^{+} \mid A \prec S\}$. 

\end{proof}

\section{Discussion}

The proofs given above can be easily adapted to masas in type $III$ factors that admit a faithful normal conditional expectation. Cartan masas, by definition satisfy this property, but not all masas do - By a result of Takesaki\cite{TakCon}, if every masa in a von Neumann algebra admits a normal conditional expectation, then it is finite. Suppose $\mathcal{A}$ is a masa in a type $III$ factor $\mathcal{M}$ admitting a normal conditional expectation $E : \mathcal{M} \rightarrow \mathcal{A}$. Let $A \in \mathcal{A}$ and $S$ be positive operators. For any self-adjoint operator $T$, let $\alpha(T) = \operatorname{min}(\{x \in \sigma(T)\}$. For any unitary $U$ in $\mathcal{M}$, we have that $||E(USU^{*})|| \leq ||S||$ and that $\alpha(E(USU^{*})) \geq \alpha(S)$. It is now easy to see that a necessary condition for the existence of an element $T \in \mathcal{O}(S)$ such that $E(T) = A$ is that $||A|| \leq ||S||$ and $\alpha(A) \geq \alpha(S)$.

 The Schur-Horn problem in type $III$ factors is more tractable that in the type $II_1$ case. Standard arguments allow us to prove the following lemma
\begin{lemma}\label{typeIIILem}
Let $S = \sum_{n=1}^{N} \mu_n F_n$ be a positive contraction with finite spectrum in a type $III$ factor $\mathcal{M}$ with $||S|| = 1$ and $\alpha(S) = 0$. Then, $\mathcal{O}(S)$ contains a non-trivial projection(and thus every projection). 
\end{lemma}

With this in hand, it is easy to see that if $A \in \mathcal{A}$ and $S \in \mathcal{M}$ are positive elements with finite spectrum so that $||A|| \leq ||S||$ and $\alpha(A) \geq \alpha(S)$, then we can solve the Schur-Horn problem for $A$ and $S$. There is further, a simple condition that allows us to determine when we can find a unitary so that $E(USU^{*}) = A$. 

Suppose $0$ is the point spectrum of $A$, so that there is a projection $P$ in $\mathcal{A}$ so that $PAP = 0$. Suppose we write $A = E(T)$ for some positive operator $T$, then, $E(PTP) = 0$ and hence, $PTP = 0$. Thus, $0$ must be in the point spectrum of $T$. If $A = E(USU^{*})$, we get that $0$ must be in the point spectrum of $USU^{*}$ and hence in the point spectrum of $S$. Similarly, if $1$ is in the point spectrum of a positive contraction $A$ and $A = E(USU^{*})$ for some positive contraction $S$ and a unitary $U$, then, $1$ must be in the point spectrum of $S$ as well. These necessary conditions are also sufficient. 

\begin{theorem}
Let $\mathcal{A}$ be a masa inside a type $III$ factor $\mathcal{M}$ admitting a faithful normal conditional expectation $E$ and let $A$ and $S$ be positive operators with finite spectrum in $\mathcal{A}$ and $\mathcal{M}$ respectively and let $E$ be the normal conditional expectation onto $\mathcal{A}$. Assume further that $\alpha(A) \geq \alpha(S)$ and $||A|| \leq ||S||$.
\begin{enumerate}
 \item There is an element $T \in \mathcal{O}(S)$ such that $E(T) = A$. 
 \item Assume additionally that if either $0$ and $||S||$ are in the point spectrum of $A$, then they are in the point spectrum of $S$ as well. Then, there is a unitary $U$ such that $E(U^{*} S U) = A$. 
\end{enumerate}
\end{theorem}

We omit the details as they are a straightforward adaptation of the proof of theorem(\ref{SHD}). 

 In general, we could ask,
\begin{question}
 Let $A$ and $S$ be positive operators in $\mathcal{A}$ and $\mathcal{M}$ respectively, where $\mathcal{A}$ is a masa inside a type $III$ factor admitting a normal conditional expectation and so that $||A|| \leq ||S||$ and $\alpha(A) \geq \alpha(S)$. Then, is there an element $T$ in $\mathcal{O}(S)$ so that $E(T) = A$?
\end{question}

Lyapunov's theorem\cite{LyapOr}, which states that the range of any non-atomic vector valued measure taking values in $\mathbb{C}^{n}$ is compact and convex, was reformulated in operator algebraic language by Lindenstrauss\cite{LinLya} to say the following: Let $\Phi$ be a weak* continuous linear map from a non-atomic abelian von Neumann algebra into $\mathbb{C}^{n}$. Then, for any positive contraction $A$, there is a projection $P$ such that $\Phi(A) = \Phi(P)$. Anderson and Akemann, in their superb monograph\cite{LyapAA}, 
called any theorem concerning linear maps $\Phi: \mathcal{X} \rightarrow \mathcal{Y}$ where $\mathcal{X}$ and $\mathcal{Y}$ are subsets of linear spaces, that assures us that $\operatorname{Ran}(\Phi) = \operatorname{Ran}(\Phi\mid \partial(\mathcal{X}))$ a Lyapunov type theorem. Clearly, Kadison's carpenter problem is a Lyapunov type problem. Anderson and Akemann proved a variety of Lyapunov theorems and showed, quite surprisingly, that Lyapunov theorems are substantially more tractable when the maps considered are singular. The one of most interest to us is
\begin{theorem}[Anderson and Akemann]
Let $\mathcal{A}$ be a masa in an type $II_1$ factor $\mathcal{M}$. Let $F$ be a singular conditional expectation from $\mathcal{M}$ to $\mathcal{A}$. Then every positive contraction can be lifted to a projection $P$ under $F$. 
\end{theorem}
There are plenty of singular conditional expectations onto masas in $II_1$ factors\cite{AkeShe}, though none of them are trace preserving. The corresponding Schur-Horn problem cannot be any other than
\begin{question}
Let $\mathcal{A}$ be a masa in an type $II_1$ factor $\mathcal{M}$. Let $F$ be a singular conditional expectation from $\mathcal{M}$ to $\mathcal{A}$.  Suppose $A \in \mathcal{A}$ and $S \in \mathcal{M}$ positive contractions that are not multiples of the identity such that $||A|| \leq ||S||$ and $\alpha(A) \geq \alpha(S)$. Then, is there an element $T \in \mathcal{O}(S)$ such that $F(T) = A$?  
\end{question}
Finally, an answer to the following related question, which we are unable to solve, should help in solving the Schur-Horn and carpenter problems in type $II_1$ factors.
\begin{question}
 Let $A$ be a positive operator in a masa $\mathcal{A}$ inside a $II_1$ factor $\mathcal{M}$. Then, does the norm closure of $\mathfrak{L}(A) = \{S \in \mathcal{M} \mid \exists \, T \in \mathcal{O}(S) \text{ so that } E(T) = A\}$ equal  $\{S \in \mathcal{M} \mid A \prec S\}$? Is $\mathfrak{L}(A)$ convex?
\end{question}

\subsection{Acknowledgements}
The authors would like to thank Sabanci University for a research grant that supported the visit of the first author to Sabanci University, Istanbul in September 2011, when part of this work was done. The second author would also like to thank Matt Daws for pointing out on Mathoverflow, a result of Takesaki that is mentioned in the last section.

\begin{multicols}{2}

\noindent {\bf B V Rajaram Bhat},\\
Indian Statistical Institute,\\
R V College Post,\\
 Bangalore-560059, India.\\
bhat@isibang.ac.in

\noindent {\bf Mohan Ravichandran},\\
Faculty of Engineering and Natural Sciences,\\
Sabanci University, Orhanli, Tuzla,\\
Istanbul, Turkey - 34956\\
mohanr@sabanciuniv.edu

\end{multicols}

\end{document}